\documentclass[final, sort&compress]{elsarticle} 
\usepackage[text={13.5cm, 22cm}, centering] {geometry} 

\usepackage{amsmath,amsthm,amsfonts,amssymb,cases}
\usepackage{enumerate, paralist}
\usepackage{CJK,CJKnumb,CJKulem}
\usepackage{xcolor}
\usepackage{graphicx} 

\journal{Journal of Finite Fields and Their Applications}

\newcommand{\Z}{\mathbb{Z}}
\newcommand{\F}{\mathbb{F}}

\newtheorem{thm}{Theorem}[section]
\newtheorem{cor}[thm]{Corollary}
\newtheorem{lem}[thm]{Lemma}

\begin{document}
\begin{CJK}{GBK}{song}
\begin{frontmatter}

\title{Piecewise constructions of inverses of cyclotomic mapping permutation polynomials\tnoteref{t1}}
\tnotetext[t1]{This work was supported by the NSF of China (Grant Nos. 11371106, 61502113, 61363069)
and the Guangdong Provincial NSF (Grant No. 2015A030310174).}

\author[GL,GZ1]{Yanbin Zheng}
\ead{zhengyanbin@guet.edu.cn}

\author[GZ1,GZ2]{Yuyin Yu}
\ead{yuyuyin@163.com}

\author[GL]{Yuanping Zhang}

\ead{ypzhang12@gmail.com}

\author[GZ1,GZ2]{Dingyi Pei\corref{cor}}
\cortext[cor]{Corresponding author.}

\ead{gztcdpei@scut.edu.cn}

\address[GL]{Guangxi Key Laboratory of Trusted Software, Guilin University of Electronic Technology,
              Guilin 541004, China}
\address[GZ1]{School of Mathematics and Information Science, Guangzhou University, Guangzhou 510006, China}

\address[GZ2]{Key Laboratory of Mathematics and Interdisciplinary Sciences of Guangdong Higher
 Education Institutes, Guangzhou University, Guangzhou 510006, China}

\begin{abstract}
  Given a permutation polynomial of a large finite field, finding its inverse is usually a hard problem. Based on a piecewise interpolation formula, we construct the inverses of cyclotomic mapping permutation polynomials of arbitrary finite fields.
\end{abstract}

\end{frontmatter}

\section{Introduction}
For $q$ a prime power, let $\F_{q}$ denote the finite field containing $q$ elements,
and $\F_{q}[x]$ the ring of polynomials over~$\F_{q}$. A polynomial $f(x) \in \F_{q}[x]$ is called a
permutation polynomial (PP) of $\F_{q}$ if it induces a bijection of $\F_{q}$.
We define a polynomial $f^{-1}(x)$ as the inverse of $f(x)$ over $\F_{q}$ if $f^{-1}(f(c)) =c$ for all $c \in \F_{q}$, or equivalently $f^{-1}(f(x)) \equiv x \pmod{x^{q} -x}$.  Given a PP $f(x)$ of $\F_{q}$, its inverse is unique in the sense of reduction modulo $x^q -x$. In theory one could use the Lagrange Interpolation Formula to compute the inverse, i.e.,
\begin{equation*}\label{Lag-1}
f^{-1}(x) =\sum_{c \in \F_{q}}c \big(1-(x-f(c))^{q-1}\big).
\end{equation*}
It is a point-by-point interpolation formula and the computing is very inefficient for large~$q$. In fact, finding the inverse of a PP of a large finite field is a hard problem except for the well-known classes such as the inverses of linear polynomials, monomials, and some Dickson polynomials. There are only several papers on the inverses of some special classes of PPs, see~\cite{MR, Wang-1} for the inverse of PPs of the form $x^rh(x^{(q-1)/d})$, \cite{Wu_L, Wu_L-1} for the inverse of linearized PPs, \cite{Coulter-1,Wu_bil} for the inverses of two classes of bilinear PPs, \cite{TW-1} for the inverses of more general classes of PPs.

The basic idea of piecewise constructions of PPs is to partition a finite field into subsets and to study the permutation property through their behavior on the subsets.
Although the idea is not new~\cite{Carlitz62, cm}, it is still currently being used to find new PPs \cite{CHZ14, LHT13, FH-PW, Hou11, Hou15, Wang-cyc, YD-AGW2, ZH12}.
In our recent work~\cite{zyp-1}, the piecewise idea is employed to construct the inverse of a large class of PPs. In Section~2, a piecewise interpolation formula for the inverses of arbitrary PPs of finite fields is presented, which generalizes the Lagrange Interpolation Formula and the result in~\cite{zyp-1}. In Section~3, using our piecewise interpolation formula, we construct the inverses of cyclotomic mapping PPs studied in~\cite{Wang-cyc}. Section~4 gives the explicit inverses of special cyclotomic mapping PPs.

\section{Piecewise constructions of PPs and their inverses}
The idea of piecewise constructions of PPs was summarized in \cite{CHZ14} by Cao, Hu and Zha,
which can also be applied to construct PPs over finite rings. For later convenience,
the following lemma expresses it in terms of finite fields.
\begin{lem}\label{PW}(See \cite[Proposition 3]{CHZ14}.)
Let $D_{1},\cdots, D_{m}$ be a partition of  $\F_{q}$,
and $f_{1}(x)$, $\cdots$, $f_{m}(x) \in \F_{q}[x]$. Define
\begin{equation}\label{f-G}
  f(x) = \sum_{i=1}^{m}f_{i}(x)I_{D_{i}}(x),
\end{equation}
where $I_{D_{i}}(x)$ is the characteristic function of $D_{i}$, i.e., $I_{D_{i}}(x) = 1$
if $x \in D_{i}$ and $I_{D_{i}}(x) = 0$ otherwise. Then $f(x)$ is a PP of $\F_{q}$ if and only if
\begin{enumerate}[$(i)$]
   \item $f_{i}$ is injective on $D_{i}$ for each $1 \le i \le m;$ and
   \item $f_{i}(D_{i}) \cap f_j(D_{j}) = \emptyset$
   for all $1 \le i \ne j \le m$.
\end{enumerate}
\end{lem}
In Lemma \ref{PW}, $f(x)$ is divided into $m$ piece functions $f_{1}(x),\cdots,f_{m}(x)$,
namely $f(x) = f_{i}(x)$ for $x \in D_{i}$. Hence $f(x)$ is a PP of $\F_{q}$ if and only if
$f_{1}(D_{1}),\cdots, f_m(D_{m})$ is a partition of $\F_{q}$.
Inspired by the lemma above, we present the following piecewise interpolation method for constructing inverses of all PPs of finite fields.

\begin{lem}\label{inv-pw}
If $f(x)$ in \eqref{f-G} is a PP of $\F_q$, then its inverse over $\F_q$ is given by
\begin{equation}\label{inv_fin set}
  f^{-1}(x) = \sum_{i=1}^{m}\bar{f_i}(x)I_{f_i(D_i)}(x),
\end{equation}
where $\bar{f_i} (f_i(c)) =c$ for $c \in D_i$, and $I_{f_{i}(D_i)}(x)$ is the characteristic function of $f_i(D_i)$.
\end{lem}
\begin{proof}
For any $c \in \F_q$, assume $c \in D_i$ for some $1 \le i \le m$, then $I_{D_{i}}(c)=1$ and $I_{D_{j}}(c)=0$ for $j \ne i$. Hence $f(c)= f_i(c) \in f_{i}(D_{i})$, $I_{f_{i}(D_{i})}(f_{i}(c))=1$ and $I_{f_{j}(D_{j})}(f_{i}(c))=0$ for $j \ne i$. Therefore $f^{-1}(f(c))=\bar{f_{i}}(f_i(c))=c$.
\end{proof}

Lemma \ref{inv-pw} gives a piecewise interpolation formula for the inverse of any PP $f(x)$;
the inverse of $f(x)$ is composed of the inverses of piece functions $f_{i}(x)$ when restricted to $D_{i}$ and the characteristic functions of $f_{i}(D_{i})$. When $m=q$, i.e., every $D_{i}$ has only one element,
the formula~\eqref{inv_fin set} is reduced to the Lagrange Interpolation Formula, and it is inefficient for large $m$. When $m$ is small and $\bar{f_{i}}(x)$ and $I_{f_{i}(D_{i})}(x)$ are known, the formula~\eqref{inv_fin set} is very efficient for any $q$.

For general $f_{i}(x)$ and $D_i$, it is difficult to find $\bar{f_{i}}(x)$ and $I_{f_{i}(D_{i})}(x)$.
But it is easy for some special cases. For instance, when every $f_{i}(x)$ is a PP of $\F_{q}$ and its inverse $\bar{f_i}(x)$ over $\F_{q}$ is known, we have proved that $I_{f_{i}(D_{i})}(x)= I_{D_{i}}(\bar{f_i}(x))$ in our previous work~\cite{zyp-1}. In this paper we remove the restriction that piece functions $f_{i}(x)$ are all PPs of $\F_{q}$, and construct inverses of cyclotomic mapping PPs by using Lemma~\ref{inv-pw}.

\section{Inverses of cyclotomic mapping permutation polynomials}

Let $\xi$ be a primitive element of $\F_{q}$, and $q-1=ds$ for some $d, s \in \Z^{+}$ (positive integers). Let the set of all $s$-th roots of unity in  $\F_{q}$ be
\[
D_0 =\{\xi^{kd} \mid k=0,1,\cdots,s-1\}.
\]
Then $D_0$ is a subgroup of $\F_{q}^{*}$, where $\F_{q}^{*}$ is the multiplication group of all nonzero elements of $\F_{q}$. The elements of the factor group $\F_{q}^{*}/D_0$ are the cyclotomic cosets
\begin{equation}\label{Di}
D_i = \xi^{i}D_0=\{\xi^{kd+i} \mid k=0,1,\cdots,s-1\}, \quad  i=0,1,\cdots,d-1,
\end{equation}
which form a partition of $\F_q^*$. For $a_{0}, \cdots, a_{d-1}\in \F_{q}$ and $r_{0},\cdots,r_{d-1} \in \Z^{+}$, a generalized cyclotomic mapping form $\F_q$ to itself is defined in~\cite{Wang-cyc} by
\begin{equation}\label{map_aixri}
f(x)=\left\{\begin{array}{ll}
 0                & \text{for $x =0$,} \\
 a_{i}x^{r_{i}}   & \text{for $x \in D_i$, $i=0,1,\cdots, d-1$.}
 \end{array}  \right.
\end{equation}
Cyclotomic mappings were introduced in~\cite{NW-CYC} when $r_{0}=\dotsm=r_{d-1} =1$ and in~\cite{Wang07} for  $r_{0}=\dotsm=r_{d-1} >1$. Further information can be found in~\cite[Section~8.1.5]{HFF}.
For some $0 \le i,k \le d-1$ and $x\in D_k$, we have $x^s=\omega^k$ where $\omega=\xi^s$, and so
\[
\sum_{j=0}^{d-1}\Big(\frac{x^s}{\omega^i}\Big)^j=\sum_{j=0}^{d-1}\omega^{(k-i)j}
=\left\{\begin{array}{ll}
 0   & \text{for $k\ne i$,} \\
 d   & \text{for $k = i$.}
 \end{array}  \right.
\]
Hence $f(x)$ in~\eqref{map_aixri} can be uniquely represented in~\cite{Wang-cyc} as
\begin{equation}\label{poly_aixri}
f(x) =\frac{1}{d}\sum_{i=0}^{d-1}a_{i}x^{r_{i}}\sum_{j=0}^{d-1}\Big(\frac{x^s}{\omega^i}\Big)^j
=\frac{1}{d}\sum_{i=0}^{d-1}\sum_{j=0}^{d-1}a_{i}\omega^{-ij}x^{r_{i} +js},
\end{equation}
in the sense of reduction modulo $x^q -x$. Theorem~2.2 in~\cite{Wang-cyc} gives several equivalent necessary and sufficient conditions for $f(x)$ in~\eqref{map_aixri} or \eqref{poly_aixri} to permute $\F_{q}$.
\begin{lem}\label{wang}
(See \cite[Theorem 2.2]{Wang-cyc}) Let $q-1 =ds$ and $d,s,r_{0}, \cdots, r_{d-1} \in \Z^{+}$. Let $a_{0},\cdots,a_{d-1}\in \F_{q}^{*}$ and $\omega =\xi^s$, where $\xi$ a primitive element of  $\F_{q}$.
Then $f(x)$ in~\eqref{poly_aixri} is a PP of $\F_{q}$ if and only if $\gcd(\prod_{i=0}^{d-1}r_{i}, s) = 1$ and $a_{i}^{s} \omega^{i r_{i}} \ne a_{j}^{s} \omega^{j r_{j}}$ for $0 \le i \ne j \le d-1$.
\end{lem}
The following lemma is also needed.
\begin{lem}\label{binomial_thm}(See \cite[Lemma 2.2]{zyp-1}.)
Let $a \in \F_{q}^*$ and $q-1 =ds$, where $d,s \in \Z^{+}$. Then
\[
1-(x^s -a^s)^{q-1} \equiv \frac{1}{d}\sum_{j=1}^{d}\Big(\frac{x^s}{a^s}\Big)^{j} \pmod{x^q-x}.
\]
\end{lem}

Now we give the main result.

\begin{thm}\label{thm_aixri}
Let the notation be as in Lemma \ref{wang}. If $f(x)$ in~\eqref{poly_aixri} is a PP of $\F_{q}$, then its inverse over $\F_{q}$ is given by
\[
f^{-1}(x)=\frac{1}{d}\sum_{i=0}^{d-1}\sum_{j=0}^{d-1}
 \omega^{i(t_i-jr_i)}\Big(\frac{x}{a_i}\Big)^{\widetilde{r_i} + js},
\]
where $\widetilde{r_i}$, $t_{i} \in \Z$ satisfy $1\le \widetilde{r_i} <s$ and $r_i\widetilde{r_i}+st_i=1$.
\end{thm}
\begin{proof}
If $f(x)$ is a PP of $\F_{q}$, then $\gcd(\prod_{i=0}^{d-1}r_{i}, s) =1$. There exist $\widetilde{r_i}$, $t_i \in \Z$ such that $r_{i}\widetilde{r_i}+st_i =1$. Next we prove that the inverse of $f(x)$ over $\F_{q}$ is
\[
f^{-1}(x) = \sum_{i=0}^{d-1}\omega^{it_{i}}(x/a_i)^{\widetilde{r_i}}
\big[1-(x^{s} - a_{i}^{s} \omega^{i r_{i}})^{q-1} \big].
\]
Let $D_i$ be defined by~\eqref{Di} and $f_{i}(x) = a_{i}x^{r_{i}}$. Then $f(x)=f_{i}(x)$ for $x \in D_{i}$. Let $\bar{f_{i}}(x) =\omega^{it_{i}}(x/a_i)^{\widetilde{r_i}}$. Then for any $c =\xi^{kd+i} \in D_{i}$ we have
\[
\bar{f_{i}} (f_{i}(c))
=\omega^{it_i}\xi^{(kd+i)r_i\widetilde{r_i}}
=(\xi^s)^{it_i}\xi^{(kd+i)(1-st_i)}
=\xi^{kd+i}=c.
\]
Now we show that the characteristic function of $f_{i}(D_{i})$ is
\[
I_{f_i(D_i)}(x) =1-(x^{s} - a_{i}^{s}\omega^{i r_{i}})^{q-1}.
\]
It follows from $\gcd(r_{i}, s)= 1$ that $\{kr_{i} \mid k=0,1,\cdots,s-1\}$ is a complete residue system modulo $s$. Therefore
\[\begin{split}
f_{i}(D_i) &= \{a_{i}\xi^{(kr_{i})d  +ir_{i}} \mid k =0,1,\cdots,s-1\} \\
&= \{a_{i}\xi^{kd +ir_{i}} \mid k=0,1,\cdots,s-1\}.
\end{split}\]
If $x=a_{i}\xi^{kd +ir_{i}} \in f_{i}(D_{i})$, then $x^s =a_{i}^s \omega^{ir_{i}}$ and so $I_{f_i(D_i)}(x)=1$.
If $x \in f_{j}(D_{j})$ and $j \ne i$, then $x^s =a_{j}^s \omega^{jr_{j}} \ne a_{i}^{s}\omega^{i r_{i}}$ since $f(x)$ is a PP of $\F_{q}$. Hence $I_{f_i(D_i)}(x)=0$. By Lemma~\ref{inv-pw}, $f^{-1}(x)$ is the inverse of $f(x)$ over $\F_{q}$.

Next we change the form of $f^{-1}(x)$. It follows from Lemma~\ref{binomial_thm} that
\[
1-\big( x^{s} - a_{i}^{s} \omega^{i r_{i}} \big)^{q-1}
\equiv \frac{1}{d}\sum_{j=1}^{d}\Big(\frac{x^s}{a_{i}^s\omega^{ir_i}} \Big)^{j} \pmod{x^q-x}.
\]
Also note that $x^{\widetilde{r_i} +ds} \equiv x^{\widetilde{r_i}} \pmod{x^q -x}$ and $(a_{i}^s\omega^{ir_i})^d=1$. Hence,
\[\begin{split}
f^{-1}(x)
&=\frac{1}{d}\sum_{i=0}^{d-1}\omega^{it_{i}}\Big(\frac{x}{a_i}\Big)^{\widetilde{r_i}}\sum_{j=1}^{d}
         \Big(\frac{x^s}{a_{i}^s\omega^{ir_i}} \Big)^{j}
 \equiv\frac{1}{d}\sum_{i=0}^{d-1}\omega^{it_{i}}\Big(\frac{x}{a_i}\Big)^{\widetilde{r_i}}\sum_{j=0}^{d-1}
         \Big(\frac{x^s}{a_{i}^s\omega^{ir_i}} \Big)^{j} \\
& \equiv \frac{1}{d}\sum_{i=0}^{d-1}\sum_{j=0}^{d-1}
 \omega^{i(t_{i}-jr_i)}\Big(\frac{x}{a_i}\Big)^{\widetilde{r_i} + js} \pmod{x^q-x}.
\end{split} \]
For different integers $\widetilde{r_i}$ and $t_i$ such that $r_{i}\widetilde{r_i}+st_i =1$, it is easy to show that $f^{-1}(x)$ is unique in the sense of reduction modulo $x^q -x$. So we may take $1\le \widetilde{r_i} <s$.
\end{proof}

\section{Application}
Special cases of the main result are considered in this section.
By Theorem~\ref{thm_aixri}, if $f(x)$ is a PP of $\F_{q}$, then $f^{-1}(x)$ has at most $d^2$ terms. When $d$ is not very large, Theorem~\ref{thm_aixri} gives an efficient method to find $f^{-1}(x)$.
The following is an example for $d=2$.

Let $q$ be odd, $s=(q-1)/2$, $a_{0}, a_{1}\in \F_{q}^{*}$ and $r_0, r_1 \in \Z^+$.
 \cite[Corollary~2.3]{Wang-cyc}~stated~that
\begin{equation}\label{f-d=2}
f(x)=\tfrac{1}{2}a_0 x^{r_0}(1+x^{s}) +\tfrac{1}{2}a_1 x^{r_1}(1-x^{s})
\end{equation}
is a PP of $\F_{q}$ if and only if $\gcd(r_0r_1, s)=1$ and $(a_0a_1)^{s}=(-1)^{r_1+1}$.
\begin{cor}\label{ri_d=2}
If $f(x)$ in \eqref{f-d=2} is a PP of $\F_{q}$, then its inverse over $\F_{q}$ is given by
\[
f^{-1}(x)=\frac{1}{2}\Big(\frac{x}{a_0}\Big)^{\widetilde{r_0}} \Big(1 +\Big(\frac{x}{a_0}\Big) ^{s}\Big)
 +\frac{1}{2}(-1)^{t_1}\Big(\frac{x}{a_1}\Big)^{\widetilde{r_1}} \Big(1 +(-1)^{r_1}\Big(\frac{x}{a_1}\Big) ^{s}\Big),
\]
where $\widetilde{r_i}$, $t_{i}\in\Z$ satisfy  $1\le \widetilde{r_i} <s$ and $r_i \widetilde{r_i}+st_i=1$.
\end{cor}

Applying Corollary \ref{ri_d=2} to $a_0 =a_1=1$, we obtain a class of self-inverse PPs.

\begin{cor}\label{selfinv2}
Let $q$ be odd and $s=(q-1)/2$. Let $r_0,r_1 \in \Z^+$ be such that $s \mid r_0^2-1$ and $2s \mid r_1^2 -1$. Then
\[
f(x)=\tfrac{1}{2}x^{r_0} (1 +x ^{s}) +\tfrac{1}{2}x^{r_1} (1 -x^{s})
\]
is a self-inverse PP over $\F_{q}$, i.e., $f(x)$ is a PP of $\F_{q}$ and $f^{-1}(x)=f(x)$.
\end{cor}
\begin{proof}
If $s \mid r_0^2-1$, then $r_0^2+st_0=1$ for some $t_0 \in \Z$. Hence $\gcd(r_0, s)=1$.
Similarly, $r_1^2 +2ms=1$ for some $m \in \Z$. Thus $\gcd(r_1, s)=1$ and $r_1$ is odd.
Now the result a direct consequence of \cite[Corollary~2.3]{Wang-cyc} and Corollary~\ref{ri_d=2}.
\end{proof}

In a symmetric cryptosystem, the decryption function is usually the same as the encryption function. Hence self-inverse PPs would be potentially useful in symmetric cryptosystems. According to the computation of mathematical software, Theorem \ref{thm_aixri} includes numerous self-inverse PPs such as $x^5+2x^3+5x$ and $2x^5+3x^3+3x$ are self-inverse PPs of $\F_7$, $x^{11}+11x^5$ and $9x^{11}+6x^8+9x^2$ are self-inverse PPs of $\F_{13}$.

Next we consider the case that $d\ge3$, $a_1= \dotsm =a_d$ and $r_1= \dotsm =r_d$.
\begin{cor}\label{r1=ri}
Let $q-1=ds$, $d\ge3, s, r_0, r_1 \in \Z^+$, $a_{0}, a_{1}\in \F_{q}^{*}$ and
\[
f(x)=(1/d)(a_0 x^{r_0} -a_1 x^{r_1})(1 +x^{s} +\dotsb +x^{(d-1)s}) +a_1 x^{r_1}.
\]
 Then $f(x)$ is a PP of $\F_{q}$ if and only if $\gcd(r_0 r_1, s)=\gcd(r_1, d)=1$ and $a_{0}^{s}=a_{1}^{s}$.
In this case the inverse of $f(x)$ over $\F_{q}$ is given by
\[
f^{-1}(x)\!=\!(1/d)\big[ (x/a_0)^{\widetilde{r_0}} -(x/a_1)^{\widetilde{r_1}}\big]
          \big[1 +(x/a_1)^{s} +\dotsb +(x/a_1)^{(d-1)s}\big] +(x/a_1)^{\widetilde{r_1}+us}.
\]
where $\widetilde{r_i}$ is the inverse of $r_i$ modulo $s$, $0 \le u < d$ and $u \equiv r'_1(1-r_1 \widetilde{r_1})/s \pmod{d}$, and $r'_1$ is the inverse of $r_1$ modulo $d$.
\end{cor}

\begin{proof}
For $D_i$ in \eqref{Di} and $x\in D_i$, $x^{s} =\omega^i$.
Since $\sum_{j=0}^{d-1}(\omega^{i})^j=0$ for $1 \le i \le d-1$,
\[ f(x) =\left\{\begin{array}{ll}
 a_0 x^{r_0}   & \text{for $x \in D_0$,} \\
 a_1 x^{r_1}   & \text{for $x \in D_i$, $1 \le i \le d-1$.}
 \end{array}  \right.
\]

(i) By Theorem 2.2 in \cite{Wang-cyc}, $f(x)$ is a PP of $\F_{q}$ if and only if $\gcd(r_0 r_1, s)=1$ and
\[
\{a_0^s, a_1^s\omega^{r_1}, \cdots, a_1^s\omega^{(d-1)r_1}\}=\{1, \omega, \cdots, \omega^{d-1}\}.
\]
Because $a_1^s$ is a power of $\omega$, the latter condition is equivalent to
\begin{equation}\label{=ud}
\{a_0^s/a_1^s,  \omega^{r_1}, \cdots, \omega^{(d-1)r_1}\}=\{1, \omega, \cdots, \omega^{d-1}\}.
\end{equation}
It is easy to show that \eqref{=ud} is equivalent to $\gcd(r_1, d)=1$ and $a_0^s/a_1^s=1$.

(ii) If $f(x)$ is a PP of $\F_{q}$, then $\gcd(r_0 r_1, s)=\gcd(r_1, d)=1$ and $a_{0}^{s}=a_{1}^{s}$. For $i=0$ or $1$, there exist $\widetilde{r_i}$, $t_{i} \in \Z$ such that $r_i \widetilde{r_i}+st_i=1$. By Theorem \ref{thm_aixri},
\[
f^{-1}(x)= (1/d)\sum_{j=0}^{d-1}(x/a_0)^{\widetilde{r_0}+js}
+(1/d)\sum_{i=1}^{d-1}\sum_{j=0}^{d-1}\omega^{i(t_1 -jr_1)}(x/a_1)^{\widetilde{r_1}+js}.
\]
Let $0 \le u < d$ and $u \equiv r'_1(1-r_1 \widetilde{r_1})/s \pmod{d}$, where $r_1r'_1 \equiv 1 \pmod{d}$. Then
\[
ur_1 \equiv r_1r'_1(1-r_1 \widetilde{r_1})/s \equiv (1-r_1 \widetilde{r_1})/s \equiv st_1/s
\equiv t_1 \pmod{d},
\]
Hence $\sum_{i=1}^{d-1}\omega^{i(t_1 -jr_1)}=-1$ for $0\le j \ne u  \le d-1$, and so
\[\begin{split}
&\quad \sum_{i=1}^{d-1}\sum_{j=0}^{d-1}\omega^{i(t_1 -jr_1)}(x/a_1)^{\widetilde{r_1}+js}
=\sum_{j=0}^{d-1}\sum_{i=1}^{d-1}\omega^{i(t_1 -jr_1)}(x/a_1)^{\widetilde{r_1}+js}\\
&=(d-1)(x/a_1)^{\widetilde{r_1}+us} -\sum_{j \ne u }(x/a_1)^{\widetilde{r_1}+js}
=d(x/a_1)^{\widetilde{r_1}+us} -\sum_{j=0}^{d-1}(x/a_1)^{\widetilde{r_1}+js}.
\end{split}\]
Also note that $a_{0}^{s}=a_{1}^{s}$. We obtain
\[\begin{split}
f^{-1}(x)
& = (1/d)\textstyle\sum_{j=0}^{d-1}(x/a_0)^{\widetilde{r_0}+js}
-(1/d)\sum_{j=0}^{d-1}(x/a_1)^{\widetilde{r_1}+js} +(x/a_1)^{\widetilde{r_1}+us} \\
& = (1/d)\textstyle\sum_{j=0}^{d-1}[(x/a_0)^{\widetilde{r_0}+js}-(x/a_1)^{\widetilde{r_1}+js}] +(x/a_1)^{\widetilde{r_1}+us} \\
& = (1/d)\textstyle\sum_{j=0}^{d-1}[(x/a_0)^{\widetilde{r_0}}(x/a_1)^{js}-(x/a_1)^{\widetilde{r_1}+js}] +(x/a_1)^{\widetilde{r_1}+us} \\
& = (1/d)\textstyle[(x/a_0)^{\widetilde{r_0}}-(x/a_1)^{\widetilde{r_1}}]\sum_{j=0}^{d-1}(x/a_1)^{js} +(x/a_1)^{\widetilde{r_1}+us}. \qedhere
\end{split}\]
\end{proof}

The first part of Corollary~\ref{r1=ri} generalizes \cite[Proposition 2.11]{Wang-cyc} which requires that $d$ is a prime divisor of $q-1$. Moreover, \cite[Proposition 2.11]{Wang-cyc} is incorrect for $l=2$, and the expression $l-1$ should be $1-l$.

Let $n,i,j \in \Z^+$ and $s=(2^{2n}-1)/3$. Corollary 2.7 in~\cite{Wang-cyc} states that
\begin{equation}\label{f-d=3}
f(x)=\big(x^{2^i} +x^{2^j}\big)\big(1 +x^{s} +x^{2s}\big) +x^{2^j}
\end{equation}
is a PP of $\F_{2^{2n}}$. The following is a direct consequence of Corollary~\ref{r1=ri}.
\begin{cor}
The inverse of $f(x)$ in~\eqref{f-d=3} over $\F_{2^{2n}}$ is given by
\[
f^{-1}(x)=\big( x^{\widetilde{2^i}} +x^{\widetilde{2^j}} \big)\big(1 +x^{s} +x^{2s}\big) +x^{\widetilde{2^j} +us},
\]
where $\widetilde{2^k}$ is the inverse of $2^k$ modulo $s$, $0 \le u < 3$ and
$u \equiv (-1)^j\big(1-2^j \widetilde{2^j} \big)/s \pmod{3}$.
\end{cor}

According to our knowledge, Wan and Lidl \cite{WL91} made the first systematic study of PPs of $\F_{q}$ of the form $f(x)=x^{r}h(x^s)$, where $q-1 =ds$, $1 \le r <s$ and $h(x) \in \F_{q}[x]$. A criterion for $f(x)$ to be a PP of  $\F_{q}$ was given in \cite{WL91}. Later on, several equivalent criteria are found in other papers; see for instance \cite{AW07,Hou15,HFF,PL01,Wang07, Zieve09}. One of the criteria is that $f(x)$ is a PP of $\F_{q}$ if and only if $\gcd(r, s)= 1$ and $x^rh(x)^{s}$ permutes $\{1, \omega,\omega^2, \cdots, \omega^{d-1}\}$, where $\omega =\xi^s$ and $\xi$ is a primitive element of $\F_{q}$. Next we employ our main result to deduce the inverse of $f(x)$.

\begin{cor}\label{xrhxs}
With the conditions and the notation introduced above, if $f(x)=x^{r}h(x^s)$ is a PP of $\F_{q}$,
then its inverse over $\F_{q}$ is given by
\[
f^{-1}(x)=\frac{1}{d}\sum_{i=0}^{d-1}\sum_{j=0}^{d-1}
         \omega^{i(t-jr)}\big(x/h(\omega^{i}) \big)^{\widetilde{r} +js},
\]
where $\widetilde{r}$, $t \in \Z$ satisfy  $1\le \widetilde{r} <s$ and $r\widetilde{r}+st=1$.
\end{cor}
\begin{proof}
For $D_i$ in \eqref{Di} and $x\in D_i$, we have $x^{s} =\omega^i$ and so $f(x) =x^{r}h(\omega^{i})$.
The proof can be completed by substituting $h(\omega^{i})$ and $r$ for $a_i$ and $r_i$ in Theorem~\ref{thm_aixri}.
\end{proof}

 Corollary \ref{xrhxs} is actually the same as Theorem 2.1 in \cite{Wang-1}.

\section*{Acknowledgments}
We are grateful to the referees for many useful comments and suggestions.

\end{CJK}
\end{document}